\documentclass[11pt,reqno]{amsart}

\usepackage{latexsym}
\usepackage{amssymb}
\usepackage{amsmath}
\usepackage{amsthm}
\usepackage{amscd}
\newtheorem{theorem}{Theorem}[section]
\newtheorem{lemma}[theorem]{Lemma}
\newtheorem{corollary}[theorem]{Corollary}
\newtheorem{proposition}[theorem]{Proposition}

\newtheorem{question}[theorem]{Question}
\theoremstyle{definition}
\newtheorem{definition}[theorem]{Definition}
\theoremstyle{remark}
\newtheorem{remark}[theorem]{Remark}
\newtheorem{example}[theorem]{Example}

\numberwithin{equation}{section}
\def\cf{\mathop{\rm cf}\nolimits}

\def\cl{\mathop{\rm cl}\nolimits}
\def\ho{\mathop{\rm Hom}\nolimits}
\def\max{\mathop{\rm max}\nolimits}

\newcommand{\N}{{\mathbb N}}

\newcommand{\Z}{{\mathbb Z}}
\newcommand{\Q}{{\mathbb Q}}
\newcommand{\h}{{\mathbf \sharp}}

\renewcommand{\P}{{\mathbb P}}

\newcommand{\T}{{\mathbb T\,}}

\newcommand{\la}{\langle\,}
\newcommand{\ra}{\,\rangle}

\newcommand{\cth}{\mathcal{T}_{\mbox{\tiny $H$}}}
\newcommand{\gd}{G_{\delta}}

\newcommand{\cc}{\mathfrak{c}}
\newcommand{\tb}{\textbf}

\date{}

\title[Pseudocompact groups]{Pseudocompact group topologies with no  infinite compact subsets}

\author{Jorge Galindo and Sergio Macario}

\address{Departmento de Matem\'aticas \\
         Universitat Jaume I \\
        Campus Riu Sec, 12071\\
        Castell\'on
        Spain}
\email{jgalindo@mat.uji.es \\ macario@mat.uji.es}

\address{}
\email{}

\subjclass[2000]{Primary 54H11,20K99, Secondary 22A10,22B05.}

\keywords{$G_\delta$-dense,  h-embedded, $\sharp$-property, compact
Abelian group, SCH, torsion-free rank, dominant rank}

\date{}
\thanks{Research   supported by the Spanish Ministry of Science (including FEDER funds), grant
MTM2008-04599/MTM and  Fundaci\'o Caixa Castell\'o-Bancaixa, grant
P1.1B2008-26.}
\begin{document}
\begin{abstract}
We show that  every Abelian group satisfying a mild cardinal
inequality   admits a pseudocompact group topology from which all
countable subgroups inherit the maximal totally bounded topology (we
say that such a topology satisfies property~$\h$).

Every pseudocompact Abelian group $G$ with  cardinality $|G|\leq 2^{2^\cc}$
satisfies this inequality and therefore admits a pseudocompact group
topology with property~$\h$. Under the Singular Cardinal Hypothesis
(SCH) this criterion can be combined with an analysis of the
algebraic structure of
  pseudocompact groups to prove that
every  pseudocompact Abelian group admits a pseudocompact group
topology with property~$\h$.

We also observe that pseudocompact Abelian groups with property~$\h$ contain
no infinite compact subsets and are examples of Pontryagin reflexive
precompact groups that are not compact.
\end{abstract}
\maketitle
\section{Introduction}

A topological space  $X$ is  pseudocompact if every real-valued
continuous function on $X$ is bounded. Pseudocompactness is greatly
enhanced by the addition of algebraic structure.  This fact was
discovered in 1966 by Comfort and Ross \cite{comfross66} who proved
that pseudocompact topological groups are totally bounded or, what
is the same, that they always appear as \emph{subgroups} of compact
groups. They  went  even further and precisely    identified
pseudocompact groups among subgroups of topological groups: a
subgroup of a compact group is pseudocompact if, and only if, it is
$\gd$-dense in its closure (i.e., meets every nonempty $\gd$-subset
of its closure).

A powerful tool to study totally bounded  topologies on Abelian groups is  Pontryagin duality.
This is because   a totally bounded group topology
is always induced by a group of
characters \cite{comfross64} and
 Pontryagin duality is  based on relating a topological
  group with its group of continuous characters.
   We recall here that  a character of a group $G$ is  nothing but  a
     homomorphism of $G$ into the multiplicative
     group $\T$ of complex numbers of
modulus one.

If $G$ is an Abelian topological group, the topology of uniform convergence on
compact subsets of  $G$ makes the  group of continuous characters of $G$,
denoted $G^\wedge$,  into a topological group.  Evaluations then define a homomorphism $\alpha_G\colon G\to G^{\wedge\wedge}
$ between $G$ and the group of all continuous characters on the dual group,
the so-called \emph{bidual} group $G^{\wedge \wedge}$.
When  $\alpha_G$
  is a topological isomorphism we say that $G$ is Pontryagin reflexive.
 It will be necessary for the development of this paper to keep in mind that
 character groups of discrete groups are compact groups. Even if it is not relevant for our purposes we cannot resist here  to add  that character groups of compact groups are again discrete, and that the Pontryagin van-Kampen theorem proves that all locally compact Abelian groups (discrete and compact ones are thus comprised) are reflexive.

 In the present paper Pontryagin duality will appear both as a tool for constructing pseudocompact group topologies and as an objective itself.
 To be  precise,  this paper is motivated by the following two questions
\begin{question}[\cite{chasmart08}]\label{Qref}
Is every Pontryagin reflexive totally bounded Abelian group a compact group?
\end{question}
\begin{question}[\cite{dikrshak05}, Question 25 of \cite{dikrshak07}]\label{Qcomp}
Does every pseudocompact Abelian group admit a  pseudocompact group topology
with no infinite compact subsets?
\end{question}
In this paper we obtain a negative answer to  Question \ref{Qref}
and a positive answer, valid under the Singular Cardinal
Hypothesis (SCH), to Question \ref{Qcomp}. The focus of the paper will be
on Question \ref{Qcomp} with the analysis of Question \ref{Qref} and
its relation with
 Question \ref{Qcomp} deferred to Section 6.

 It should be noted, in a direction opposite to Question
 \ref{Qcomp}, that every pseudocompact group admits pseudocompact
 group topology with nontrivial convergent sequences, see
 \cite{galigarctomi09}.

Our approach to Question \ref{Qcomp} consists in combining
techniques that can be traced back at least to \cite{tkac88} with
the ideas of \cite{galigarc07}.
  Our
construction actually produces pseudocompact Abelian groups with all
countable subgroups $h$-embedded. This is stronger (see Section 2)
that finding pseudocompact group topologies with no infinite compact
subsets. With the aid of results from \cite{hernmaca03} this
construction will yield a wide range of negative answers to Question
\ref{Qref}. As  pointed to  us by M. G. Tkachenko, Question
\ref{Qref} has been answered independently in \cite{ardaetal}.
\subsection*{On notation and terminology}
All groups considered in this paper will be Abelian. So, the
specification \emph{Abelian group} to be found at some points will
respond only to a matter of emphasis. To further avoid the
cumbersome use of the word "Abelian", free Abelian groups will
simply be termed as  \emph{free groups}.

The symbol $\P$ will denote the set of all prime numbers.
\emph{Faute de mieux}, we will use the unusual symbol
$\mathbb{P}^{\uparrow}$ to denote the set of all prime powers, i.e.,
an integer  $k\in \mathbb{P}^{\uparrow}$ if, and only if, $k=p^n$ for
some $p\in \P$ and some positive integer $n$.

For a set $X$ and a cardinal number $\alpha$, $[X]^{\alpha}$ stands for
the collection of all subsets of $X$ with cardinality $\alpha$.

Following Tkachenko \cite{tkac88}, we say that a subgroup $H$ of a
topological group $G$ is $h$-embedded if every homomorphism of $H$
to the unit circle $\T$ can be extended to a \emph{continuous}
homomorphism of $G$ to $\T$. If $G$ is totally bounded and $H$ is
$h$-embedded in $G$, then the topology of $H$ must equal the maximal
totally bounded topology of $H$ (or, using van Douwen's terminology,
$H=H^\sharp$).

The cardinal function $m(\alpha)$ will be often used. The cardinal
$m(\alpha)$ is defined for every infinite cardinal $\alpha$ as the
least cardinal number of a $\gd$-dense subset of a compact group
$K_\alpha$ of weight $\alpha$. It is proved in \cite{comfrobe85}
that this definition does not depend on the choice $K_\alpha$ and
therefore makes sense. The same reference contains proofs of the
following basic essential features of $m(\alpha)$:
\[ \log(\alpha)\leq m(\alpha)\leq (\log(\alpha))^\omega \quad \mbox{
and }\quad  \cf(m(\alpha))>\omega, \quad \mbox{ for every
}\alpha\geq \omega.\] These inequalities have a much simpler form if
\emph{Singular Cardinal Hypothesis} (SCH) is assumed. SCH is a
condition consistent with ZFC that follows from (but is much weaker
than) the \emph{Generalized Continuum Hypothesis} (GCH). Under SCH
every infinite cardinal $\alpha$ satisfies
\[
 m(\alpha)=(\log(\alpha))^\omega.
\]
It is well known that every compact group has cardinality $2^\kappa$
for some cardinal $\kappa$. The question on which cardinals can
appear as the cardinal of a pseudocompact group is not so readily
answered. We will say that a cardinal $\kappa$ is \emph{admissible}
provided there is a pseudocompact group of cardinal $\kappa$. The
first obstructions to admissibility were found by van Douwen
\cite{douw80}, the main one being  that the cardinality $|G|$ of a
pseudocompact group cannot be a strong limit cardinal of countable
cofinality; see \cite[Chapter 3]{dikrshak98} for more information on
admissible cardinals.

Most of our results concern constructing  pseudocompact group
topologies on a given Abelian group $G$. As indicated in the
introduction, every pseudocompact group topology is totally bounded
and a totally bounded group topology $\mathcal{T}$ on an Abelian
group $G$ is always induced by a unique group of characters
$H\subset Hom(G,\T)$, \cite{comfross64,comfross66}. To stress this
latter fact we will usually refer to $\mathcal{T}$ as
$\mathcal{T}_{_{H}}$. Recall that the topology $\mathcal{T}_{_{H}}$
is Hausdorff if, and only if, the subgroup $H$ separates points of
$G$.

We have also introduced above the symbol $G^\wedge$ to denote the group
of all continuous characters of a topological Abelian group equipped with
the compact-open topology. We will use in this context the subscript $_d$
 to indicate that $G$ carries the discrete topology. Thus $(G_d)^\wedge$
equals the set $Hom(G,\T)$ of all homomorphisms into $\T$.
Being a closed subgroup of $\T^G$, $(G_d)^\wedge$ is always a compact group.

Several  purely algebraic notions from the theory of infinite
Abelian groups will be necessary, as for instance the notion of
basic subgroup and the related one of pure subgroup. We refer to
\cite{fuchs} for  the meaning and significance of these properties.
As usual, the symbol $t(G)$ stands for the torsion subgroup of the
group $G$ and $r_0(G)$ denotes the torsion-free rank of $G$.

\section{The dual property to
pseudocompactness} The following theorem is at the heart of the
relationship between questions \ref{Qcomp} and \ref{Qref}.
 \begin{theorem}[\cite{hernmaca03}]\label{hernmaca}
 Let $(G,\cth)$, $H\subset \ho(G,\T)$, be a Hausdorff Abelian totally bounded group.
 $(G,\cth)$ is pseudocompact if, and only if, every countable subgroup of
$(H,\mathcal{T}_{_{G}})$ is $h$-embedded in $(G_d)^\wedge$.
 \end{theorem}

\begin{definition}
We say that a topological group $G$ has property $\h$ if every
countable subgroup of $G$  is $h$-embedded in $G$.
\end{definition}
Thus  property $\h$ is, in the terminology of \cite{hernmaca03},
the dual property of pseudocompactness.

The relation between property $\h$ and Question \ref{Qcomp} is clear
from the following Lemma. Although a  combination of
Propositions~3.4 and 4.4 of \cite{hernmaca03} would  provide an
indirect proof,  we  offer a  direct proof for the reader's
convenience.
\begin{lemma} \label{nocompact}
  Let $(G,\cth)$ denote a totally bounded group with property $\h$.
   Then $(G,\cth)$ has no infinite compact subsets.
\end{lemma}
\begin{proof}
We first see that all countable subgroups of $G$ are $\cth$-closed.
Suppose  otherwise that $x\in \cl_{(G,\cth)}N\setminus N$ with $N$ a
countable subgroup of $G$. The subgroup $\widetilde{N}=\la N\cup
\{x\}\ra$ is also countable and, by hypothesis, inherits its maximal
totally bounded group topology from $(G,\cth)$. Since subgroups are
necessarily closed in that topology, it follows that $N$ is closed in
$\widetilde{N}$, which goes against $x\in \widetilde{N}\setminus N$.

Now suppose $K$ is an infinite  compact subset of $G$ and let
$S\subset K$ be a countable subset of $K$. Define $\widetilde{G}=\la
S\ra$ and denote by $\widetilde{G}$  and
$\overline{(\widetilde{G},\cth)}$ the completions of
$\widetilde{G}^\sharp$ and  $(\widetilde{G},\cth)$ respectively.
 Since $\la S\ra$ is $h$-embedded the identity
function $j\colon \widetilde{G}^\sharp \to (\widetilde{G},\cth)$,
extends to a topological isomorphism $\bar{\j}\colon
b\widetilde{G}\to \overline{(\widetilde{G},\cth)}$. Then
$\bar{\j}(\cl_{b\widetilde{G}}S)=\cl_{\overline{(\widetilde{G},\cth)}}j(S)\subset
K$, therefore
$\cl_{\overline{(\widetilde{G},\cth)}}j(S)=\cl_{(\widetilde{G},\cth)}
S$ and, it follows from the preceding paragraph that
$\cl_{b\widetilde{G}}S=\bar{\j}(\cl_{b\widetilde{G}}S)\subset \la
S\ra$.

But a  well known theorem of van Douwen \cite{douw90} (see also
\cite{galihern98} and \cite[Theorem 9.9.51]{arhatkac} for different
proofs and \cite{galihern99fu} for extensions of that result) states
that $|\cl_{b(\widetilde{G})} S|=2^\cc$ and therefore it is
impossible that $\bar{\j}(\cl_{b\widetilde{G}}S)S\subset \la S \ra$.
\end{proof}

We establish next some easily deduced permanence properties.
\begin{proposition}\label{prop:h-prod}
The class of groups having property $\h$ is closed for finite
products.
\end{proposition}
\begin{proof}
Let $G_1$ and $G_2$ be two topological Abelian groups with property
$\h$ and let $N$ be a countable subgroup of $G_1\times G_2$.  Let
$h$ be a homomorphism from $N$ to $\mathbb{T}$. By considering an
arbitrary extension of $h$ to $G_1\times G_2$ we may assume that $h$
is actually defined on $G_1\times G_2$. Since both $\pi_1(N)$ and
$\pi_2(N)$  are countable there will be continuous homomorphisms
$h_i\colon G_i \to \T$, $i=1,2$, with $h_1(x)=h(x,0)$ and
$h_2(y)=h(0,y)$ for all $x\in \pi_1(N)$ and $y \in \pi_2(N)$. The
homomorphism $\bar{h}\colon G_1\times G_2\to \T$ given by
$\bar{h}(x,y)=h_1(x)\cdot h_2(y)$  is  then a continuous extension
of $h$.
\end{proof}

\begin{lemma}\label{lem:quotients}
 Let $\pi:K\to L$ be a continuous surjection between two compact Abelian groups $K$
and $L$
  and suppose that  $N$ is  a subgroup of $L$
  that, as subspace  of $L$, carries the
maximal totally bounded topology. If $M$ is a subgroup of $K$ such
that $\pi_{\upharpoonleft_M}$ is  a group  isomorphism between $M$
and $N$, then $M$ also inherits from $K$ the maximal totally bounded
topology.
\end{lemma}
  \begin{proof}
Denote by  $\mathcal{T}_{_{K}}$ and $\mathcal{T}_{_{L}}$ the
topologies that $M$ inherit from $K$ and $L$  respectively (the
latter obtained through $\pi_{\upharpoonleft_M}$). Since $\pi$ is
continuous, the topology
 $\mathcal{T}_{_{K}}$ is finer than $\mathcal{T}_{_{L}}$, but
 $\mathcal{T}_{_{K}}$ is the maximal totally bounded topology,
therefore $\mathcal{T}_{_{K}}=\mathcal{T}_{_{L}}$.
  \end{proof}

\section{Property $\h$ on torsion-free and bounded groups}
We will make a heavy use of powers of groups in the sequel. If $\sigma$ is a
cardinal number, $K^\sigma$ stands for such powers. We use calligraphical letters,  to denote
sets of coordinates, that is,  subsets of $\sigma$. If $\mathcal{D}\subset \sigma$,
we will denote by $\pi_{\mathcal{D}}^{K}$ the
projection from $K^\sigma$ to $K^{\mathcal{D}}$, if  no confusion is
possible we will simply use $\pi_{\mathcal{D}}$.

\begin{lemma} \label{room}
Let $G$ be a metrizable group and let $\sigma \geq \cc$ and $\alpha$
be cardinal numbers  with $m(\sigma)\leq \alpha$, and
$\alpha^\omega\leq \sigma$.

Then  there exists an independent  $\gd$-dense subset $D\subseteq
G^\sigma$ with  cardinality $m(\sigma)$, $D=\{d_\eta \colon \eta
<m(\sigma)\}$, and two  families of sets of coordinates $\{
\mathcal{S}_\theta \colon \theta \in [\alpha]^\omega\}, \{
\mathcal{N}_\eta\colon \eta< \alpha\}\subset \sigma$  such that:
\begin{enumerate}
\item $|\mathcal{S}_\theta|=\sigma$.
\item $\mathcal{S}_\theta\cap\mathcal{S}_{\theta^\prime}=\emptyset$, if
 $\theta\neq \theta^\prime$.
\item ${\displaystyle \left|\mathcal{S_\theta}\setminus
\bigcup_{\eta \in \theta}\mathcal{N}_\eta\right|=\sigma}$ for every
 $\theta\in [\alpha]^{ \omega}$.
\item Every subset $\{g_\eta \colon \eta <\alpha\}$ of $G^\sigma$ with
$\pi_{_{\mathcal{N}_\eta}}(g_\eta)=\pi_{_{\mathcal{N}_\eta}}(d_\eta)$,
for all $\eta <\alpha$  is $\gd$-dense.
\end{enumerate}
\end{lemma}
\begin{proof}
Let $\mathcal{A}_\beta=\{a_{\gamma}\colon \gamma<\sigma\}$ be a set
with $|\mathcal{A}_\beta|=\sigma$ and consider the disjoint union
$\mathcal{A}=\bigcup_{\beta<\cc} \mathcal{A}_\beta$. We identify
$G^\sigma$ with $G^\mathcal{A}$ and $\alpha$ with $[\cc]^\omega
\times \alpha$. Since $\alpha^\omega \leq \sigma$, we can as well
decompose each $\mathcal{A}_\beta$ as a disjoint union
${\displaystyle \mathcal{A}_\beta= \bigcup_{\widetilde{\theta}\in
[[\cc]^\omega\times \alpha]^\omega}
\mathcal{A}_{\beta,\widetilde{\theta}}}$ of sets of cardinality
$|\mathcal{A}_{\beta,\widetilde{\theta}}|=\sigma$.

For each $N\in [\cc]^\omega$, let next $F_N=\{f_{(N,\eta)}\colon
\eta<\alpha\}$ be an independent  $\gd$-dense subset of  the product
${\displaystyle G^{\cup_{\gamma\in N} \mathcal{A}_\gamma}}$ (note
that $m(\sigma)\leq \alpha$ and that $G$ is metrizable). Assume that
each $f_{(N,\eta)}$ actually belongs to $G^\mathcal{A}$ by putting
$\pi_{\mathcal{A}_\gamma}(f_{(N,\eta)})=0$ if $\gamma\notin N$.

We now order $\alpha=[\cc]^\omega \times \alpha$ lexicographically
and define the sets $N_{\widetilde{\eta}}$, $\widetilde{\eta} \in
[\cc]^\omega \times \alpha$ and $\mathcal{S}_{\widetilde{\theta}}$,
$\widetilde{\theta} \in [[\cc]^\omega \times \alpha]^\omega$.
 For $\widetilde{\eta}=(N,\eta)\in [\cc]^\omega \times \alpha$  define
$\mathcal{N}_{(N,\eta)}=\bigcup_{\gamma\in N}
\mathcal{A}_{\gamma,\widetilde{\eta}}$
 and given
$\widetilde{\theta}=\{(N_k,\eta_k)\colon k<\omega,\;\:
(N_k,\eta_k)\in [\cc]^\omega \times\alpha\}$, we define
$\mathcal{S}_{\widetilde{\theta}}=\mathcal{A}_{\beta_0,\widetilde{\theta}}$
where $\beta_0$ is such that $\beta \in N_k$ for some $k$, implies
$\beta<\beta_0$ (recall that $\cc$ has uncountable cofinality). By
construction of the sets $\mathcal{A}_{\beta,\widetilde{\theta}}$,
we have
$\mathcal{S}_{\widetilde{\theta}}\cap\mathcal{S}_{\widetilde{\theta}^\prime}=\emptyset$,
when $\widetilde{\theta}\neq \widetilde{\theta}^\prime$. Condition
(3) obviously holds, since $S_{\widetilde{\theta}}$ and
$\bigcup_{\widetilde{\eta}\in\widetilde{\theta}}\mathcal{N}_{\widetilde{\eta}}$
are even disjoint.

Define finally $D=\{ f_{\widetilde{\eta}}\colon \widetilde{\eta}\in
[\cc]^\omega  \times \alpha\}=\cup_{N\in[\cc]^\omega}F_N$.

Suppose $\widetilde{D}=\{g_{\widetilde{\eta}}\colon
\widetilde{\eta}\in [\cc]^\omega \times \alpha\}$ is such that
$\pi_{_{\mathcal{N}_{\widetilde{\eta}}}} (g_{\widetilde{\eta}})=
\pi_{_{\mathcal{N}_{\widetilde{\eta}} } } (f_{\widetilde{\eta}})$,
for all $\widetilde{\eta} \in [\cc]^\omega \times \alpha$.

 To check that $\widetilde{D}$ is
indeed $\gd$-dense we choose a $\gd$-subset $U$ of
 $G^\mathcal{A}$ .
There will be then $N=\{\alpha_n \colon n<\omega\}\in [\cc]^\omega$
and  a $\gd$-set $V \subset G^{\cup\mathcal{A}_{\alpha_n}}$ such
that $\{ \bar{x}\in G^{\mathcal{A}}\colon
\pi_{\cup_n\mathcal{A}_{\alpha_n}} (\bar{x})\in V \mbox{ for each
}n<\omega\}\subset U$. Since $F_N$ is $\gd$-dense in
$G^{{\displaystyle \cup_{\gamma\in N}}
\mathcal{A}_\gamma}=G^{{\displaystyle \cup_n
\mathcal{A}_{\alpha_n}}}$, there will be
  an element $f_{(N,\eta)}\in F_N$ with
 $\pi_{\cup_n \mathcal{A}_{\alpha_n}}(f_{(N,\eta)})\in V$ for every $\alpha_n \in N$.

As $ g_{(N,\eta)}$ and $f_{(N,\eta)}$ have the same $\cup_{\gamma
\in N}\mathcal{A}_\gamma$-coordinates, we conclude that
$g_{(N,\eta)}\in U\cap \widetilde{D}$.
\end{proof}

If $\chi$ is  a homomorphism  between two groups $G_1$ and $G_2$ and
 $\sigma$ is a cardinal number, we denote by $\chi^{\sigma}$  the
 product homomorphism
 $\chi^\sigma \colon G_1^{\sigma}\to G_2^{\sigma}$ defined by
$\chi^{\sigma}((g_\eta)_{\eta<\sigma})=(\chi(g_\eta))_{\eta<\sigma}$.
It is easily verified that, for any $\mathcal{D}\subseteq\sigma$,
the projections $\pi_{\mathcal{D}}^{G_i}:G_i^{\sigma}\to
G_i^{\mathcal{D}}$, $i=1,2$ satisfy
\[
\pi_{\mathcal{D}}^{G_2}\circ\chi^{\sigma}=\chi^{\mathcal{D}}\circ \pi_{\mathcal{D}}^{G_1}
\]

\begin{corollary}\label{cor:room}
  Let $\chi \colon G_1\to G_2$ be a surjective homomorphism between two
   metrizable
  groups $G_1$ and $G_2$. If $\sigma$ and $\alpha$ are  cardinal numbers
  with $m(\sigma)\leq \alpha$ and $\alpha^\omega\leq \sigma$, then
  it is possible to find an independent  $\gd$-dense subset $D$ of
  $G_1^\sigma$  satisfying the properties of Proposition
  \ref{room} such that in addition $\chi^\sigma(D)$ is an
  independent subset of $G_2^\sigma$.
\end{corollary}
\begin{proof}
  It suffices to repeat the proof  of Lemma \ref{room} taking care to
  choose the sets $F_N$ in such a way that $\chi^{\cup_{\gamma \in
  N}\mathcal{A}_\gamma}(F_N)$ is also independent.
\end{proof}

\begin{proposition}\label{prop:freetorsion}
Let $\chi \colon G\to \T$ be a surjective character of  a compact
metrizable group $G$. If $\sigma $ and $\alpha$ are cardinal numbers
with $m(\sigma)\leq \alpha$, and $\alpha^\omega\leq \sigma$,  then
the topological group $G^\sigma$ contains an independent $\gd$-dense
subset $F$ of cardinality $\alpha$
 such that $F$ and $\chi^\sigma (F)$ generate isomorphic groups  with
 property $\h$.
\end{proposition}
 \begin{proof}
We begin with a $\gd$-dense  subset   of $G^\sigma$, $D=\left\{
d_\eta \colon \eta<\alpha\right\}$,  with the properties of Lemma~\ref{room}
and Corollary~\ref{cor:room}. We have thus
two families of sets $\{\mathcal{S}_\theta,\: \colon\: \theta\in[\alpha]^{\omega}\}$,
$\{\mathcal{N}_\eta,\: \colon\: \eta<\alpha\}\subset\sigma$
with the properties (1) through (4) of that Lemma.

Next, for every $\theta\in [\alpha]^{\omega} $, we choose and fix a
set of coordinates $\mathcal{D}_\theta \subseteq \sigma$ of
cardinality $|\mathcal{D}_\theta|=\sigma$ in such a way that
\[\mathcal{D}_\theta\subseteq \mathcal{S}_\theta
 \setminus {\displaystyle\bigcup_{\eta \in \theta}
 \mathcal{N}_\eta}\]
(recall that  by Lemma \ref{room},
$\left|\mathcal{S_\theta}\setminus \bigcup_{\eta \in
\theta}\mathcal{N}_\eta\right|=\sigma$)

Given each $\theta\in[\alpha]^{\omega}$, we consider the free
subgroup
  $\left\langle  \chi^\sigma(d_{\eta})
\colon \eta \in \theta \right\rangle$ and equip it with its maximal
totally bounded topology. Denoting  the resulting topological  group
as $\left\langle  \chi^\sigma( d_{\eta}) \colon \eta \in \theta
\right\rangle^\sharp$, and taking into account that it has weight
$\cc$, we can find   an embedding
\begin{equation}\label{(1)}
j_{\theta} \colon \left\langle
\chi^\sigma(d_{\eta}) \colon \eta \in \theta \right\rangle^{\sharp}
\hookrightarrow \T^{\mathcal{D}_\theta}.
\end{equation}

 For each $\theta \in [\alpha]^\omega$ and each $\eta\in \theta$,
 let
$g_{\eta,\theta}$ denote an element of $G^\mathcal{D_\theta}$ with
$\chi^{\mathcal{D}_\theta}(g_{\eta,\theta})=j_\theta(\chi^\sigma(d_\eta))$.
Observe that the set $\{ g_{\eta,\theta}\colon \eta \in \theta\}$ is
independent.

We finally define the elements $f_\eta$, $\eta<\alpha$,  by the
rules: \begin{align*} &\pi_{\mathcal{D}_\theta}^G(f_\eta)=
g_{\eta,\theta} \mbox{, if } \theta \in [\alpha]^\omega \mbox{ is
such that } \eta \in \theta, \qquad \mbox{ and }\\ &
\pi_{\gamma}^G(f_\eta)=\pi_\gamma^G(d_\eta) \mbox{ if }  \gamma
\notin  \mathcal{D}_\theta \mbox{ for any } \theta \in
[\alpha]^{\omega} \mbox{ with } \eta \in \theta. \end{align*}

 Let us see that $F=\{f_\eta\colon\eta<\alpha\}$
satisfies the desired properties:
\begin{enumerate}
\item \emph{$F$  and $\chi^\sigma(F)$ are  independent.} Suppose
that $\sum_{k=1}^m n_k f_{\eta_k}=0$ with $n_k\in \Z$. Choose then
$\theta \in [\alpha]^\omega$ with $\eta_1,\ldots , \eta_{m}, \in
\theta$. Since
$\pi_{\mathcal{D}_\theta}^G(f_{\eta_k})=g_{{\eta_k},\theta}$ and the
set $\{g_{\eta,\theta} \colon \eta \in \theta\}$ is independent, the
independence of $F$ follows. Since
$\pi_{\mathcal{D}_\theta}(\chi^\sigma(f_{\eta}))=\chi^{\mathcal{D}_\theta}(g_{\eta,\theta})$,
$\chi^\sigma(F)$ is also independent. It is easy to see, now, that
$\langle F\rangle$ and $\langle\chi^\sigma(F)\rangle$ are isomorphic.

 \item \emph{The subgroup $\langle \chi^\sigma(F)\rangle$ has property $\h$.}
Let $N$ be a countable subgroup of $\la\chi^\sigma(F)\ra$. Let
$\theta\in [\alpha]^\omega$ be  such that $N\subseteq \la
\chi^\sigma(f_\eta) \colon \eta \in \theta\ra$ and define
$N_\theta:=\la f_\eta \colon \eta \in \theta \ra$.


Observe  finally that $\pi_{\mathcal{D}_\theta}^{\T}(N)=
\chi^{\mathcal{D}_\theta}(\pi_{\mathcal{D}_\theta}^G(N_\theta))$.
This last subgroup is just $j_\theta\left(\left\langle \chi^\sigma(
d_{\eta}) \colon \eta \in \theta \right\rangle\right)$
 and the latter carries by construction its maximal totally bounded
 topology, since
 the restriction of  $\pi_{\mathcal{D}_\theta}^{\T}\colon
\T^{\sigma}\to\T^{\mathcal{D}_\theta}$  to $N$ is a group
isomorphism onto $\pi_{\mathcal{D}_\theta}^{\T}(N)
=\chi^{\mathcal{D}_\theta}(\pi_{\mathcal{D}_\theta}^G(N_\theta))$,
  Lemma~\ref{lem:quotients}   applies.

\item \emph{$\langle F\rangle $ has property $\h$.}
Take $\pi=\chi^{\sigma}$, $K=G^\sigma$ and $ L=\T^\sigma$. Bearing
in mind that the restriction to $\la F\ra$ is an isomorphism because
$F$ and $\chi^{\sigma}(F)$ are independent sets,
  Lemma~\ref{lem:quotients} applies again.
\item \emph{$F$  is a  $G_\delta$-dense subset of $G^\sigma$.}
Observe that, for every $\eta<\alpha$,  $f_\eta$ coincides with
$d_\eta$ on the set of coordinates  $\mathcal{N}_\eta $,  for
$\mathcal{D}_\theta\subseteq \mathcal{S}_\theta
 \setminus {\displaystyle\bigcup_{\eta \in \theta} \mathcal{N}_\eta}$.
 Since $D$ has the properties of Lemma~\ref{room},
 we conclude that $F$ is  $\gd$-dense.
\end{enumerate}
\end{proof}
\begin{proposition}\label{prop:boundedtorsion}
Let $\sigma $  and $\alpha$ be cardinal numbers  with $m(\sigma)\leq
\alpha$, and $\alpha^\omega\leq \sigma$. The topological group
$\Z(p)^{\sigma}$ contains an independent $\gd$-dense subset $H$ with
property $\h$.
\end{proposition}
 \begin{proof}
Proceed  exactly as  in Proposition~\ref{prop:freetorsion} and
construct an embedding into $\Z(p)^\sigma$. To obtain the
$\h$-property we identify countable subgroups with Bohr groups of
the form $\left(\oplus_{\omega} \Z(p)\right)^\sharp$.
 \end{proof}

\section{The algebraic structure of pseudocompact Abelian groups}

We obtain here some results on the algebraic structure of
pseudocompact that will be useful in the next section.  The first of
them is   inspired (and shares a part of its proof) from the first
part of the proof of Lemma 3.2 of \cite{galigarc07}. We sketch here
the proof for the reader's convenience. We  thank Dikran Dikranjan
for pointing a misguiding sentence in a previous version of this
proof.
\begin{lemma}\label{descomp}
Every Abelian group admits a decomposition \[ G=\left(
\bigoplus_{p^k \in \P_0^\uparrow} \bigoplus_{\gamma(p^k)}\Z(p^k)
\right)\bigoplus H \] where $ \P_0^\uparrow $ is a finite subset of
$\P^\uparrow$  and $H$ is a subgroup of $G$ with
\[ \left| n H\right|=|H|, \mbox{ for all }n\in \N.\]
\end{lemma}

\begin{proof}
Decompose $t(G)=\bigoplus_p G_p$ as a direct sum of $p$-groups
$G_p$ and let $B_p$ denote a basic subgroup of $G_p$ for each
$p$. This in particular means that $B_p$ is a direct sum
of cyclic $p$-groups,
\[ B_p = \bigoplus_{n < \omega} B_{p,n} \mbox{  with }
B_{p,n} \cong \bigoplus_{\beta_{p^n}}\tb Z(p^n)\] and that $G_p/B_p$
is divisible.
%
%
Define  $\mathcal{D}=\{|B_{p,n}|\colon p^n \in
\mathbb{P}^\uparrow\}$. If $\mathcal{D}$  has no maximum or
$\beta_0=\max \mathcal{D}$ is attained at an infinite number of
$|B_{p,n}|$'s we stop here. If, otherwise, $\beta_0=\max
\mathcal{D}=|B_{p_1,n_1}|=\ldots=|B_{p_r,n_r}|$ and
$|B_{p_j,n_j}|<\beta_0$ for all the remaining  $p_j^{n_j}\in
\P^\uparrow$ we  repeat the process with the set
$\mathcal{D}\setminus |B_{p_1,n_1}|$. After a finite number of steps
we obtain in this manner a  finite collection of cardinals $F\subset
\mathcal{D}$ such that either:
\begin{enumerate}
\item \emph{Case 1:}
the supremum $\beta:=\sup\left( \mathcal{D}\setminus F\right)$  is
not attained, or
\item \emph{Case 2:}
the supremum $\beta:=\sup \left(\mathcal{D}\setminus F\right)$ is
attained infinitely often,  i.e., there is an infinite subset
$I\subset \P^{\uparrow}$ with $|B_{p,n}|=\beta$ for all $p^n \in I$.
\end{enumerate}
Define  $\P_0^{\uparrow}=\{ p^n \in \P^\uparrow \colon  |B_{p,n}|\in
F\}$ (observe that $\P_0^\uparrow$ is necessarily  finite),  and set
$\gamma(p_k^{n_k})=|B_{p_k,n_k}|$ if $p_k^{n_k}\in \P_0^{\uparrow}$.
Since the subgroups $B_{p_k,n_k}$ are bounded pure subgroups,
there will be \cite[Theorem 27.5]{fuchs} a subgroup $H$ of $G$ such
that
\[ G=\left(\bigoplus_{p_k^{n_k}\in \P_0^{\uparrow}}
\bigoplus_{\gamma(p_k^{n_k})}B_{p_k,n_k}\right) \bigoplus H, \]

 For each prime $p$, consider a $p$-basic subgroup
$B_{p,H}=\oplus_{n} B_{p,n,H}$ of $H_p$, the $p$-part of $t(H)$, it
is immediately checked that either $B_{p,H}$ itself  (if $p\not \in
\P_0^\uparrow$) or $B_{p,H} \bigoplus \left(\bigoplus
_{\overset{p_k^{n_k}\in \P_0^{\uparrow}}{p_k=p}}
\bigoplus_{\gamma(p_k^{n_k})}B_{p_k,n_k}\right)$ (if $p \in
\P_0^\uparrow$) is  also $p$-basic in $G$.

Since different basic subgroups are necessarily isomorphic
\cite[Theorem 35]{fuchs}, we have that $B_{p,H}$ or $B_{p,H}
\bigoplus \left(\bigoplus_{\overset{p_k^{n_k}\in
\P_0^{\uparrow}}{p_k=p}}
\bigoplus_{\gamma(p_k^{n_k})}B_{p_k,n_k}\right)$ is isomorphic to
$B_p$. We have therefore that, for each $p$, either $\sup
|B_{p,n,H}|$ is not attained (case 1 above) or   attained at
infinitely many $p^n$'s (case 2).

Let now $n$ be any natural number. Then $|n
B_{p_k,n_k,H}|=|B_{p_k,n_k,H}|$ unless $p_k^{n_k}$ divides $n$.
Since this will only happen for finitely $p_k^{n_k}$'s, we conclude,
in both cases 1 and 2 that  $|nB_{p,H}|=|B_{p,H}|$.

Using that $B_{p,H}$ is pure in  $H_p$ and that $H_p/B_{p,H}$ is
 divisible we have that,
\begin{align*}
|nH_p|&=\left|\frac{nH_p}{nB_{p,H}}\right|+\left|nB_{p,H}\right| \\
&=
\left|n\left(\frac{H_p}{B_{p,H}}\right)\right|+\left|B_{p,H}\right|\\
&=
\left|\frac{H_p}{B_{p,H}}\right|+\left|B_{p,H}\right|=|H_p|.\end{align*}
Since $ |H|=\sum_p H_p + r_0(H)|$ for every infinite group $H$ and
$r_0(nH)=r_0(H)$ we have finally that $|H|=|nH|$, for every $n\in
\Z$.
\end{proof}

The terminology  introduced in the next definition is motivated,  in
the present context, by Theorem \ref{estrpseu} below.
\begin{definition} \label{def:split}If $G$ is an Abelian group, the set
$\P_0^\uparrow$ of  Lemma \ref{descomp} can be partitioned as
$\P_0^\uparrow=\P_1^\uparrow\cup\P_2^\uparrow$ with $p_i^{n_i}\in
\P_1^\uparrow$ if, and only if, $\gamma(p_i^{n_i})> r_0(G)$.

The cardinal numbers $\gamma(p_i^{n_i})$ with $p_i^{n_i} \in
\P_1^\uparrow$ will be called the \emph{dominant ranks} of $G$.
\end{definition}
\begin{lemma}
  \label{estrdikgio}
If $G$ is a nontorsion pseudocompact group, then there is a positive
integer such that:
 \begin{equation}\label{ineqwd}m(w(nG))\leq r_0(nG)\leq 2^{w(nG)}.\end{equation}
\end{lemma}
\begin{proof}
  If $nG$ is metrizable for some $n\in \N$, then $nG$ is a
  compact metrizable group. Therefore $r_0(nG)= \cc$ and
  the inequalities in \eqref{ineqwd} hold for this $n$.

If $nG$ is not metrizable for any $n\in \N$, then $G$ is, in the
terminology of \cite{dikrgior08}, \emph{nonsingular}. Combining
Lemma 3.3 and Theorem 1.15 of \cite{dikrgior08}, there must be $n\in
N$ such that $r_0(nG)$  is the cardinal of a pseudocompact group of
weight $w(nG)$. Therefore
 \[m(w(nG))\leq r_0(nG)\leq 2^{w(nG)}.\]
\end{proof}
\begin{theorem}\label{estrpseu}
Let $G$ be an Abelian group. If $G$ admits a pseudocompact group
topology, then $G$ can be decomposed as
\[ G=
\left(\bigoplus_{ p^k \in \P_1^\uparrow}  \bigoplus_{\gamma(p^k)}
\Z(p^k)\right) \oplus G_0\] where $\gamma(p_i^{k_i})$, $p_i^{k_i} \in
\P_1^\uparrow$, are the dominant ranks of $G$ and there
is a cardinal $\omega_d(G)$  such that
\begin{equation}
  \label{eq:estr}
  m(\omega_d(G))\leq r_0(G)\leq |G_0|\leq 2^{\omega_d(G)}.
\end{equation}
\end{theorem}
\begin{proof}
Since every pseudocompact torsion group must be of bounded order,
the theorem is trivial (and vacuous) for such groups, we may assume
that $G$ is nontorsion.

 Decompose $G$ as in Lemma \ref{descomp}:
\[\left( \bigoplus_{p^k \in \P_0^\uparrow} \bigoplus_{\gamma(p^k)}\Z(p^k) \right)\bigoplus H  \]
with $\P_0^{\uparrow}$ a finite subset of $\P^\uparrow$ and
\[ \left| n H\right|=|H| \mbox{ for all }n\in \N.\]

Split $\P_0^\uparrow=\P_1^\uparrow\cup \P_2^\uparrow$ as in
Definition \ref{def:split} and define
\[
G_0=\bigoplus_{p_i^{k_i} \in \P_2^\uparrow}
\bigoplus_{\gamma(p_i^{k_i})}\Z(p_i^{k_i})\bigoplus H.
\]
We will prove that the inequalities \ref{eq:estr} hold for
$w_d(G)=w(nG_0)$.

Lemma  \ref{estrdikgio}  proves   that there is  some $n\in \N$ with
 \begin{equation}\label{inwd} m(w(nG_0))\leq r_0(G_0)\leq
 2^{w(nG_0)}.\end{equation}
If $|G_0|=\gamma(p_i^{k_i})$ for some $p_i^{k_i} \in \P_2^\uparrow$,
it follows from the definition of $P_2^\uparrow$ that $|G_0|=
r_0(G)$ and \eqref{eq:estr} is deduced from \eqref{inwd}. If,
otherwise, $|G_0|=|H|$, then $|nG_0|\geq |nH|=|H|= |G_0|$ and we
deduce that $|G_0|=|nG_0|$ and thus that $|G_0|\leq 2^{w(nG_0)}$.
This together with \eqref{inwd} gives again \eqref{eq:estr} with
$w_d(G)=w(nG_0)$.
\end{proof}
\begin{remark}
  The cardinal $w_d(G)$ used  in  Theorem
   \ref{estrpseu} is precisely the \emph{divisible weight} of $G$ that was   introduced and studied by Dikranjan and Giordano-Bruno \cite{dikrgior08}. We refer the reader to that paper  to get an idea of the important
   role played by the divisible weight in the structure of pseudocompact groups.
One of its applications (Theorem 1.19 loc. cit.) is  to prove that
  $r_0(G)$ is an admissible cardinal for every pseudocompact group $G$, a fact first proved by  Dikranjan and Shakhmatov in \cite{dikrshak09}.
\end{remark}
%
\section{Pseudocompact groups with property $\h$}

The results of the previous sections will be used here to obtain
sufficient  conditions for    the existence of pseudocompact group
topologies with property $\h$.
\begin{lemma}  \label{enumer}
Let $\pi \colon G_1\to G_2$ be a quotient homomorphism between two
Abelian  topological groups $G_1$ and $G_2$  and let $L$ be a
compact Abelian group. Assume that the following conditions hold:
\begin{enumerate}
\item $G_1$ contains a free $\gd$-dense
  subgroup $H_1$ such that $H_1$ and $\pi(H_1)$
are isomorphic and have property $\h$.
\item $G_1$ contains another free subgroup $H_2$ such that
$H_1\cap H_2=\{0\}$, $H_1+H_2$ and  $\pi(H_1+H_2)$ are isomorphic
and  have property $\h$.
\item $m(w(L))\leq |H_2|$.
\end{enumerate}
Under these conditions the product $G_1\times L$ contains
 a $\gd$-dense subgroup $\widetilde{H}$ such that both
$\widetilde{H}$ and   $\pi\left(p_1(\widetilde{H})\right)$ have
property $\h$, where
 $p_1\colon G_1\times L\to G_1$ denotes the first projection.
\end{lemma}
\begin{proof}
We first enumerate the elements of $H_1$ and $H_2$ as
$H_1=\{f_\beta\colon \kappa <\beta\}$ and $H_2=\{ g_\eta \colon \eta
<\alpha\}$. Since $m(w(L))\leq \alpha=|H_2|$, we can also enumerate
 a $\gd$-dense subgroup $D$ of $L$ (allowing repetitions if
 necessary) as $D=\{d_\eta \colon \eta <\alpha\}$.
 We now define the subgroup $\widetilde{H}$ of $G_1\times L$ as
 \[ \widetilde{H}= \left\langle \,(f_\kappa + g_\eta,d_\eta)
 \colon \eta <\alpha, \kappa<\beta\,\right\rangle.\]
 It is  easy to check that $\widetilde{H}$ is a $\gd$-dense
 subgroup of $G_1\times L$ with $\widetilde{H}\cap \{0\}\times
 L=\{(0,0)\}$.

Since the homomorphism $p_1$ is continuous and establishes a group
isomorphism between $\widetilde{H}$ and $H_1+H_2$,  Lemma
\ref{lem:quotients} shows that
 $\widetilde{H}$  has
property~$\h$. The  same argument applies to the group   $\pi\left(
p_1(\widetilde{H})\right)=\pi(H_1+H_2)$.
\end{proof}
\begin{definition}
Let $\alpha\geq \omega$ be a cardinal. We say that \emph{$\alpha$
satisfies property~$(\ast)$ if:}
\begin{equation}
\tag{*}\label{*} \mbox{there is a  cardinal $\kappa$ with }
\kappa^\omega \leq \alpha\leq 2^\kappa
\end{equation}
 \end{definition}
Every cardinal $\alpha$ with  $\alpha^\omega = \alpha$ satisfies
property~\eqref{*}. This condition is equivalent to the condition
$(m(\alpha))^\omega \leq \alpha$.

To apply Lemma \ref{enumer} we need the following result:

\begin{theorem}[Theorem 4.5 of \cite{comfgali03}]\label{comfgali}
Let $G=(G,\mathcal{T}_1)$ be a pseudocompact Abelian group with
$w(G)=\alpha>\omega$, and set
$$\sigma=\min\{r_0(N):N \text{ is a closed $\gd$-subgroup of $G$}\}.$$
 If $\alpha^\omega\leq\sigma$ and
if $\lambda\geq\omega$ satisfies $m(\lambda)\leq\sigma$, then $G$
admits a pseudocompact group topology $\mathcal{T}_2$ such that
$w(G,\mathcal{T}_2)=\alpha+\lambda$ and
$\mathcal{T}_1\bigvee\mathcal{T}_2$ is pseudocompact. Moreover,
every closed $\gd$-subgroup of $(G,\mathcal{T}_1)$ is $\gd$-dense
$(G,\mathcal{T}_2)$.
\end{theorem}
\begin{corollary}\label{cor:comfgali}
Let $\sigma,\alpha$ and $\lambda$ be cardinals with
$\alpha^\omega\leq \sigma$ and $m(\lambda)\leq \sigma$. If
 $H$ is a  free, dense  subgroup of $\T^\sigma$ with property $\h$ and
cardinality $\alpha$,
 then $\T^\sigma $ contains another subgroup
$H_2$ with $H\cap H_2=\{0\}$, $|H_2|=\lambda+\alpha$ and  such that
$H+H_2$ has property $\h$.
\end{corollary}
\begin{proof}
Let $F(\sigma)$ denote the free Abelian group of rank $\sigma$. We
apply Theorem~\ref{comfgali} to the pseudocompact group
$(F(\sigma),\cth)$ defined  by $H$.
 We  obtain thus a pseudocompact
topology $\mathcal{T}_{_{H_2}}$ on $F(\sigma)$ induced by a subgroup
$H_2$ of $\T^\sigma$ of cardinality $|H_2|=\alpha+\lambda$ such that
$\cth \bigvee \mathcal{T}_{_{H_2}}=\mathcal{T}_{_{H+H_2}}$ is
pseudocompact. By Theorem \ref{hernmaca} the subgroup $H+H_2$ has
property $\h$ and, since closed $\gd$-subgroups of $\cth$ are
$\gd$-dense in $\mathcal{T}_{_{H_2}}$, we  also have that $H\cap
H_2=\{0\}$.\end{proof}
\begin{theorem}\label{main}
Let $G$ be  a pseudocompact Abelian group with dominant ranks
$\gamma(p_1^{n_1}),\ldots, \gamma(p_k^{n_k})$ and suppose that
$\gamma(p_i^{n_i})$, $1\leq i\leq k$, satisfy property~\eqref{*}. If
 $r_0(G)$ also satisfies property \eqref{*} for some $\kappa$ with  $m(|G_0|)\leq 2^\kappa$,
     then $G$ admits a
pseudocompact topology with property $\h$.
\end{theorem}
\begin{proof}
Decompose, following Theorem~\ref{estrpseu}, $G$ as a direct sum
\[G=\left(\bigoplus_{\gamma(p_1^{n_1})}\Z(p_1^{n_1})\bigoplus \cdots
\bigoplus_{\gamma(p_k^{n_k})}\Z(p_k^{n_k})\right) \bigoplus G_0\]

Let $F$ denote a free Abelian group of cardinality $r_0(G)$
contained in $G_0$ and denote by $D(F)$ and $D(t(G_0))$ divisible
hulls of $F$ and $t(G_0)$, respectively. There is then a chain of
group embeddings (here we use \cite[Lemmas 16.2 and 24.3]{fuchs})

\begin{equation}  \label{basicemb}
 F \overset{j_1}{\to} G_0 \overset{j_2}{\to} D(F)\oplus D(t(G_0))
\end{equation}

Denote by $\chi$  the quotient homomorphism obtained as the dual map
of the canonical embedding $\Z\to\Q$. Observe that identifying $F$
with $\oplus_{r_0(G)}\Z$ and $D(F)$ with $\oplus_{r_0(G)}\Q$, the
dual map of $j_2\circ j_1$ is exactly $\chi^{r_0(G)}$.

  Taking $\sigma=r_0(G)$, $G=\Q_d^\wedge$ and $\alpha=\kappa^{\omega}$, we can apply Proposition \ref{prop:freetorsion} to get a
$\gd$-dense   subgroup $H_1$ of
$\left(D(F)_d\right)^\wedge=\Bigl(\Q_d^\wedge\Bigr)^{r_0(G)}$ with $|H_1|=\kappa^\omega$
and such that $H_1$ and $\chi^{r_0(G)}(H_1)$ are isomorphic and have
property $\h$ (notice that $\kappa^\omega $ and $r_0(G)$ satisfy the
hypothesis of that Proposition).

We now apply Corollary \ref{cor:comfgali} to $\chi^{r_0(G)}(H_1)$ to
obtain another free  subgroup $H_2^\prime$ of $\T^{r_0(G)}$ with
$\chi^{r_0(G)}(H_1)\cap H_2^\prime=\{0\}$, $|H_2^\prime|=2^\kappa$
and such that $\chi^{r_0(G)}(H_1)+H_2^\prime$ has property $\h$. By
lifting (through $\chi^{r_0(G)}$) the free generators of
$H_2^\prime$ to $(D(F)_d)^\wedge$, we obtain a free subgroup $H_2$ of
$(D(F)_d)^\wedge$  such that $H_1\cap H_ 2=\{0\}$ and $|H_2|=
2^\kappa$. Clearly  $H_1+H_2$ is isomorphic to
$\chi^{r_0(G)}(H_1)+H_2^\prime$ and therefore $H_1+H_2$ has property
$\h$ by Lemma \ref{lem:quotients}.

We finally apply Lemma \ref{enumer}. The role of $G_1\times L$ is
played by $(D(F)_d)^\wedge\times \biggl(D(t(G_0))_d\biggr)^\wedge$; $G_2$ is here
identified with $\T^{r_0(G)}$ and $\pi$ is $\chi^{r_0(G)}$. Lemma~\ref{enumer}
then  provides a $\gd$-dense subgroup $\widetilde{H}$
  of  $\biggl(D(F)_d\biggr)^\wedge\times \biggl(D(t(G_0))_d\biggr)^\wedge$ such that
both $\widetilde{H}$ and   $\chi^{r_0(G)}(p_1(\widetilde{H}))$ have
property $\h$. This subgroup generates a pseudocompact topology
$\mathcal{T}_{_{\widetilde{H}}}$  on    $D(F)\oplus D(t(G_0))$
with property $\h$ that makes $F$ pseudocompact (the induced topology on $F$ is just the topology
inherited from $\chi^{r_0(G)}(p_1(\widetilde{H}))$). Since $G_0$ sits between $F$ and
$D(F)\oplus D(t(G_0))$, it follows that the
restriction of $\mathcal{T}_{_{\widetilde{H}}}$ to $G_0$ is
pseudocompact and has property~$\h$.

By Proposition \ref{prop:boundedtorsion} the bounded group
${\displaystyle \bigoplus_{\alpha(p_1^{n_1})}\Z(p_1^{n_1})\bigoplus
\cdots \bigoplus_{\alpha(p_k^{n_k})}\Z(p_k^{n_k})}$ also admits a
pseudocompact group topology with property $\h$ and the theorem
follows.
\end{proof}

Dikranjan and Shakmatov \cite{dikrshak05} prove under a
set-theoretic axiom called $\nabla_\kappa$ (that implies
$\cc=\omega_1$ and $2^\cc=\kappa$ with $\kappa$ being any cardinal
$\kappa\geq \omega_2$) that every pseudocompact group  of
cardinality at most $2^\cc$  has a pseudocompact group topology with
no infinite compact subsets. It follows from Theorem~\ref{main} that
the result is true in ZFC, even for larger cardinalities.
\begin{corollary}
Let $G$ be a pseudocompact Abelian group of cardinality $|G|\leq 2^{2^\cc}$.
Then $G$ admits a pseudocompact topology with property $\h$ (and
thus  a pseudocompact topology with no infinite compact subsets).
\end{corollary}
\begin{proof}
Since  a pseudocompact group with $  r_0(G)<\cc$ is
   a bounded  group
it will suffice to check that  every cardinal $\alpha$ with
$\alpha\leq 2^{2^{\cc}}$ satisfies property \eqref{*}. Theorem
\ref{main} will then be applied. We consider the following two
cases:

\emph{Case 1:  $\cc\leq \alpha \leq 2^\cc$}. In this case we put
$\kappa =\cc$.

\emph{Case 2:  $\alpha>2^\cc$}. Choose $\kappa =2^\cc$ for this
case.

Observe that in both cases $|m(|G|)|\leq 2^\kappa$ and hence that
all hypothesis of Theorem \ref{main} are fulfilled.
 \end{proof}
 By van Douwen's theorem \cite{douw80},
   a strong limit admissible cardinal  must have uncountable cofinality.
Under  mild set-theoretic  assumptions this implies that admissible
cardinals must have property \eqref{*}. It suffices, for instance,
to assume the \emph{Singular Cardinal Hypothesis} SCH.


\begin{theorem}[Theorem 3.5 of \cite{comfremus93} and  Lemma
3.4 of \cite{dikrshak98}] \label{th:SCH} If SCH is assumed, then
every admissible cardinal has property \eqref{*}.
\end{theorem}

 Combining  Theorem~\ref{estrpseu} and Theorem~\ref{main},  it turns out that, under SCH,
   every pseudocompact group admits a
pseudocompact group topology with property $\h$.
\begin{theorem}[SCH]\label{maingch}
Every  pseudocompact Abelian group $G$ admits a pseudocompact group topology with property
  $\h$.
\end{theorem}
\begin{proof}
 Let $\gamma(p_1^{n_1})\geq\cdots\geq \gamma(p_k^{n_k})$ be
 the dominant ranks of $G$. Then $|G|=\gamma(p_1^{n_1})$ and,
$  \gamma(p_1^{n_1})$ is admissible. Since  we can assume that
$n_i<n_j$ when $j>i$ and $p_i=p_j$, $p_1 G$ will be  a pseudocompact
group of cardinality $|p_1 G|=\gamma(p_2^{n_2})$.  Proceeding in the
same way we obtain that the dominant ranks are  admissible
cardinals. By Theorem \ref{th:SCH} all these cardinals must satisfy
property~\eqref{*}.
 Theorem \ref{estrpseu} shows,  on the other hand, that
 the cardinal $r_0(G)$ is also admissible and, actually:
\[
 m(w_d(G_0))\leq r_0(G_0)=r_0(G)\leq |G_0|\leq 2^{w_d(G_0)}
\]
In order to apply
Theorem~\ref{main} and finish the proof,
we must show that $r_0(G)$ also satisfies property~\eqref{*} for some cardinal $\kappa$ with  $m(|G_0|)\leq 2^{\kappa}$.

We have two possibilities:

\emph{Case 1:  $m(w_d(G_0)) \leq r_0(G)\leq (w_d(G_0))^{\omega}$}. In this case, we put
$\kappa =\log(w_d(G_0))$. Then, bearing in mind that, under SCH, we have $m(\alpha)=(\log(\alpha))^{\omega}$ for every infinite cardinal $\alpha$, we get:
\[
 \kappa^{\omega}=\biggl(\log\bigl(w_d(G_0)\bigr)\biggr)^{\omega}=
m\bigl(w_d(G_0)\bigr)\leq r_0(G)
\]
and
\[
 r_0(G)\leq  \bigl(w_d(G_0)\bigr)^{\omega}\leq \biggl(2^{\log\bigl(w_d(G_0)\bigr)}\biggr)^{\omega}=(2^{\kappa})^{\omega}= 2^{\kappa}.
\]
So property~\eqref{*} is checked. On the other hand,
\[
 m(|G_0|)\leq m\bigl(2^{w_d(G_0)}\bigr)=\biggl(\log\bigl(2^{w_d(G_0)}\bigr)\biggr)^{\omega}\leq \bigl(w_d(G_0)\bigr)^{\omega}\leq 2^{\kappa}
\]

\emph{Case 2:  $\bigl(w_d(G_0)\bigr)^{\omega}\leq r_0(G)\leq 2^{w_d(G_0)}$}. In this case, property~\eqref{*} and condition $m(|G_0|)\leq 2^{\kappa}$ are obviously fulfilled with $\kappa =w_d(G_0)$.
\end{proof}
Theorem \ref{maingch}  relies quite strongly on SCH. It uses the
construction of Theorem \ref{main} made applicable to all admissible
cardinals by  Theorem \ref{th:SCH}. We do not know whether SCH is
essential for Theorem \ref{maingch}, i.e., whether the theorem is
true for pseudocompact groups  whose cardinal does not satisfy
property~\eqref{*}.

Indeed, admissible cardinals not satisfying property~\eqref{*} are
hard to find in the literature. The following (consistent) example,
suggested to us by W.W. Comfort and based on a construction due to
Gitik and Shelah, produces  one such cardinal. We refer to Remark
3.14 of the forthcoming paper \cite{comfgotc} for additional remarks
concerning the Gitik-Shelah models. This same paper contains related
results concerning the cardinals $m(\alpha)$ and, more generally,
the density character  of powers of discrete groups in the
$\kappa$-box topology.
\begin{example}\label{gitikshel}
A pseudocompact group $G$  whose cardinality does not satisfy property~\eqref{*}.
\end{example}
\begin{proof}
Gitik and Shelah, \cite{gitikshel},  construct a model where
$m(\aleph_\omega)=\aleph_{\omega+1}$ while
$2^{\aleph_\omega}=(\aleph_{\omega})^\omega=\aleph_{\omega+2}$. This
means that the compact group $\{1,-1\}^{\aleph_\omega}$ has a
$\gd$-dense subgroup $G$ of cardinality $|G|=\aleph_{\omega +1}$.
Let us denote for simplicity $\alpha=\aleph_{\omega +1}$.

Suppose that $\alpha$ satisfies property \eqref{*}. There is then
 a
cardinal $\kappa$ with
\begin{equation}\label{2}\kappa^\omega\leq \alpha \leq 2^\kappa.\end{equation}
Since $\alpha^\omega \geq( \aleph_\omega)^\omega=
\aleph_{\omega+2}>\alpha$, we see that $\kappa^\omega\neq \alpha$.
It follows then from \eqref{2} that $\kappa^\omega \leq
\aleph_{\omega} \leq 2^\kappa$. But then $m(\aleph_\omega)\leq
m(2^\kappa)\leq \kappa^\omega\leq \aleph_\omega$, whereas, by
construction, $m(\aleph_\omega)=\aleph_{\omega+1}$. This
contradiction shows that $\alpha$ does not satisfy property
\eqref{*}.
\end{proof}
 \section{Property $\h$ and the duality of totally bounded Abelian
 groups}
   Pontryagin duality was designed to work in  locally compact Abelian groups
    and usually  works better for complete groups.
This behaviour raised the question (actually our first motivating
Question \ref{Qref}) as to whether all totally bounded reflexive
group should be compact,
 \cite{chasmart08}. We see next  that this is not the case.
 \begin{theorem}
 If a
 pseudocompact Abelian group  contains
  no infinite compact subsets, then it  is Pontryagin reflexive.
 \end{theorem}
 \begin{proof}
Let $G=(G,\cth)$ be a pseudocompact group with no infinite compact
subsets. The group of continuous characters of $G$ is then precisely
$H$ and since $G$ has no infinite compact subsets, the topology of
this dual group will equal the topology of pointwise convergence on
$G$, therefore $G^\wedge=(H,\mathcal{T}_{_{G}})$ (see in this
connection \cite{racztrig01}). By Theorem \ref{hernmaca},
$(H,\mathcal{T}_{_{G}})$ must be again a totally bounded group with
property $\h$ and hence with no infinite compact subsets, the same
argument as above then shows that $G^{\wedge\wedge}=
\left(H,\mathcal{T}_{_{G}}\right)^{\wedge}=(G,\cth)$ and therefore
that $G$ is reflexive.
\end{proof}
This last theorem combined with Lemma \ref{nocompact} and the
results of Section~5 provides a wide range of examples that answer
negatively  Question \ref{Qref}. This question has also been
answered independently in \cite{ardaetal} where another collection
of examples has been obtained.
\begin{corollary}[SCH]
Every infinite pseudocompact Abelian group $G$
supports a noncompact,  pseudocompact group topology $\cth$ such
that $(G,\cth)$ is reflexive. \end{corollary}
\begin{corollary} Every infinite
pseudocompact Abelian group $G$ with $|G|\leq 2^{2^\cc}$ supports a
noncompact,  pseudocompact group topology $\cth$ such that
$(G,\cth)$ is reflexive.
\end{corollary}
\subsection*{Acknowledgements}
We heartily thank  M.G. Tkachenko  for sharing with us a preprint
copy of \cite{ardaetal} and D. Dikranjan for his remarks on a
previous version of this paper and for making us aware of Lemma
\ref{estrdikgio}. We are also indebted to  W. W. Comfort and to D.
Dikranjan for their help concerning Example \ref{gitikshel}.
\def\cprime{$'$} \def\cprime{$'$} \def\cprime{$'$} \def\cprime{$'$}
  \def\polhk#1{\setbox0=\hbox{#1}{\ooalign{\hidewidth
  \lower1.5ex\hbox{`}\hidewidth\crcr\unhbox0}}}
  \def\polhk#1{\setbox0=\hbox{#1}{\ooalign{\hidewidth
  \lower1.5ex\hbox{`}\hidewidth\crcr\unhbox0}}} \def\cprime{$'$}

 \end{document}